\documentclass[11pt]{article}
\usepackage{algorithm,algorithmic,amsmath,amssymb,epsfig,epstopdf}
\usepackage[a4paper]{geometry}
\usepackage{url}
\usepackage{layout}
\usepackage{rotating}

\usepackage{caption}
\usepackage{subcaption}

\usepackage{color}
\usepackage{gauss}
\usepackage{graphicx}

\newcommand{\veq}{\mathrel{\rotatebox{90}{$=$}}}

\def\antisymm{\mathop{\rm anti}\nolimits}

\newcommand{\bfA}{{\textbf{A}}}
\newcommand{\bfX}{{\textbf{X}}}
\newcommand{\calX}{{\mathcal X}}

\newcommand{\calA}{{\mathcal A}}
\newcommand{\calB}{{\mathcal B}}

\newcommand{\calY}{{\mathcal Y}}
\newcommand{\calZ}{{\mathcal Z}}
\newcommand{\calS}{{\mathcal S}}

\newcommand{\calT}{{\mathcal T}}

\newcommand{\R}{{\mathbb R}}

\newcommand{\mb}[1]{\left[\begin{array}{#1}}
\newcommand{\me}{\end{array}\right]}
\newcommand{\smb}{\left[\begin{smallmatrix}}
\newcommand{\sme}{\end{smallmatrix}\right]}

\newcommand{\rank}{\text{rank}}

\numberwithin{equation}{section}
\newtheorem{theorem}{Theorem}[section]

\newtheorem{remark}[theorem]{Remark}
\newtheorem{lemma}[theorem]{Lemma}
\newtheorem{corollary}[theorem]{Corollary}

\author{Erna Begovi\'c\ Kova\v{c}\footnote{Faculty of Chemical Engineering and Technology, University of Zagreb, 10000 Zagreb, Croatia. {\tt ebegovic@fkit.hr} The author has been supported in part by Croatian Science Foundation under the project 3670.}
\and Daniel Kressner\footnote{MATHICSE-ANCHP, EPF Lausanne, 1015 Lausanne, Switzerland. {\tt daniel.kressner@epfl.ch}}
}

\title{Structure-preserving low multilinear rank approximation  \\ of antisymmetric tensors}

\begin{document}

\maketitle

\textit{Key words:} tensor, antisymmetric, low rank, singular value decomposition, Jacobi rotation.

\vspace{1ex}
\textit{AMS subject classification.} 15A69, 65F99

\begin{abstract}
This paper is concerned with low multilinear rank approximations to antisymmetric tensors, that is, multivariate arrays for which the entries change sign
when permuting pairs of indices. We show which ranks can be attained by an antisymmetric tensor and discuss the adaption of existing approximation algorithms to preserve antisymmetry, most notably a Jacobi algorithm. Particular attention is paid to the important special case when choosing the rank equal to the order of the tensor. It is shown that this case can be addressed with an unstructured rank-$1$ approximation. This allows for the straightforward application
of the higher-order power method, for which we discuss effective initialization strategies.
\end{abstract}

\section{Introduction}

A tensor $\calA \in \R^{n\times \cdots \times n}$ of order $d\ge 2$  is called \emph{antisymmetric}
if its entries $\calA(i_1,i_2,\ldots, i_d)$ change sign when permuting pairs of indices.
For example, a tensor of order three with entries $\calA(i_1,i_2,i_3)$ is antisymmetric if
\[
 \calA(i_1,i_2,i_3) = -\calA(i_2,i_1,i_3) = -\calA(i_3,i_2,i_1) = -\calA(i_1,i_3,i_2), \qquad
 i_1,i_2,i_3 = 1,\ldots, n.
\]
A direct consequence of antisymmetry is that the entries $\calA(i_1,i_2,\ldots, i_d)$ of $\calA$ satisfy
\begin{equation}\label{notzero}
\calA(i_1,i_2,\ldots,i_d)=0, \quad \text{if} \  i_p=i_q, \ \text{for some} \ p\neq q, \ 1\leq p,q\leq d.
\end{equation}
For order two, the notion of antisymmetric tensors coincides of course with the notion of skew-symmetric matrices.

Antisymmetric tensors play a major role in quantum chemistry, where the Pauli exclusion principle
implies that wave functions of fermions are antisymmetric under
permutations of variables. This antisymmetry needs to be taken into account
when solving the multiparticle Schr\"odinger equation determining such a wave function; see~\cite{Szalay2015}
for a recent overview.

This paper is concerned with finding an approximation $\calB$ to a given
antisymmetric tensor $\calA$ such that $\calB$ has a data-sparse representation and is again antisymmetric.
More specifically, we will consider an approximation of multilinear rank $r$ in structure-preserving Tucker decomposition
\begin{equation} \label{eq:tuckerb}
 \calB = \calS \times_1 U \times_2 U \cdots \times_d U,
\end{equation}
where $\calS \in \R^{r\times \cdots \times r}$ for some $r \le n$ is again antisymmetric and $U \in \R^{n\times r}$ has orthonormal columns.
This choice is analogous to existing approaches for symmetric tensors, see, e.g.,~\cite{DeLathauwer2000, LathauwerLowrank2000,Kofidis2002}.
In this paper, we demonstrate that some existing algorithms for the symmetric case extend to the antisymmetric case. In particular,
we study the extension of the Jacobi algorithm by Ishteva, Absil, and Van Dooren~\cite{IshtevaJac2013}.

Despite a number of similarities, there are pronounced differences between symmetric and antisymmetric tensors.
For example, every (multilinear) rank $r$ can be attained by a symmetric matrix or tensor.
In contrast, it is well known that skew-symmetric matrices have even rank.
Although this statement does not extend to $d>2$, we will see that there are still restrictions on the
ranks that can be attained by anti-symmetric tensors.
In particular, the smallest possible nonzero rank is $r = d$. In this case, the decomposition~\eqref{eq:tuckerb}
simplifies to
\begin{equation} \label{eq:rank1}
 \antisymm\big(\alpha u_1 \otimes u_2 \otimes \cdots \otimes u_d \big), \quad \alpha \in \R,
\end{equation}
with the antisymmetrizer $\calA = \antisymm(\calX)$ defined by
\begin{equation} \label{eq:antisymmetrizer}
 \calA(i_1,\ldots,i_d) := \frac{1}{d!} \sum_{\pi \in S_d} \text{sign}(\pi) \calX\big(\pi(i_1),\pi(i_2),\ldots,\pi(i_d)\big),
\end{equation}
where $S_d$ denotes the symmetric group on $\{1,\ldots,d\}$. This corresponds to the notion of Slater determinants that
feature prominently in the Hartree-Fock method from quantum mechanics.
The expression~\eqref{eq:rank1} suggests the more general decomposition
$\antisymm(\calX)$ for a (non-symmetric) tensor $\calX$ of low \emph{tensor} rank. This
corresponds to a short sum of Slater determinants used, e.g., in the Multi-Configuration Self-Consistent Field method.
Such a low-rank model for antisymmetric tensors
has been studied in the literature. In particular, Beylkin, Mohlenkamp, and P\'erez~\cite{Beylkin2008,Beylkin2005} have
developed an alternating least-squares algorithm for approximating
a given antisymmetric tensor $\calA$ by $\antisymm(\calX)$. The algorithm employs L\"owdin's rule to avoid having to deal with the exponentially
many terms in the sum~\eqref{eq:antisymmetrizer}. One contribution of this paper is a much simpler approach for~\eqref{eq:rank1}, that is, when $\calX$ has rank $1$: The best choice of $\calX$ is given by a scalar multiple of the best (non-symmetric) rank-$1$ approximation of $\calA$. For multilinear rank $r$ larger than $d$, our developments deviate because the low-rank decomposition~\eqref{eq:tuckerb} differs from the CP-like decomposition considered in~\cite{Beylkin2008,Beylkin2005}. Working with~\eqref{eq:tuckerb} comes with a number of advantages, such as the possibility to obtain robust quasi-optimal approximation via the SVD. Its major disadvantage is the need for storing the core tensor $\calS$ of order $d$. This can be mitigated by using hierarchical SVD-based decompositions, such as the tensor train decompositions~\cite{Oseledets2011a}. However, the incorporation of antisymmetry into these decompositions is by no means as seamless as for the Tucker decomposition; see~\cite{Hackbusch2016anti} for recent progress. Although not unlikely. it remains to be seen whether the developments in this paper are useful in this context.

The rest of this paper is organized as follows. In Section~\ref{sec:multilinrank}, we study the multilinear rank
of an antisymmetric tensor and recall the higher-order singular value decomposition.
Section~\ref{sec:multilinrankapprox} is concerned with algorithms that aim at the antisymmetric
low multilinear rank approximations, the higher-order iterations method and a variant of the Jacobi method.
Section~\ref{sec:rankd} is dedicated to the special case of rank-$d$ approximation.

\section{Multilinear rank of antisymmetric tensors}\label{sec:multilinrank}

Let us first recall some basic concepts related to the multilinear rank of a tensor; see~\cite{Kolda2009} for details.
For any $1\leq \mu\leq d$, the $\mu$th \emph{matricization} of a general tensor $\calX\in\R^{n_1\times n_2\times\cdots\times n_d}$ is the $n_\mu\times\prod_{\nu\neq \mu}n_\nu$ matrix $\bfX_{(\mu)}$ defined by
\begin{equation} \label{eq:defmat}
  \bfX_{(\mu)}(i_\mu,j) = \calX(i_1,\ldots,i_d), \qquad j = j(i_1,\ldots,i_d) := 1 + \sum_{\nu = 1\atop \nu\not=\mu}^d (i_\nu -1) \prod_{\eta = 1\atop \eta \not= \mu}^{\nu-1} n_\eta.
\end{equation}
The \emph{multilinear rank} of $\calX$ is the tuple $(r_1,r_2,\ldots,r_d)$ defined by $r_\mu = \rank(\bfX_{(\mu)})$. Note that $\bfX_{(\mu)}$ is a matrix and hence $r_\mu \leq \min\{ n_\mu, \prod_{\nu\neq \mu}n_\nu \}$.

For an antisymmetric tensor, all matricizations are essentially the same.

\begin{lemma}\label{lemma:antimat}
Let $\calA\in\R^{n\times n\times\cdots\times n}$ be an antisymmetric tensor of order $d$.
Then $\bfA_{(\mu)} = (-1)^{|\mu - \nu|} \bfA_{(\nu)}$ holds for any $1 \le \mu,\nu \le d$.
\end{lemma}

\begin{proof}
Without loss of generality, let $\mu \le \nu$.
According to~\eqref{eq:defmat}, $\bfX_{(\mu)}(i_\mu,j) = \calX(i_1,\ldots,i_d)$ implies
$\bfX_{(\nu)}(i_\mu,j) = \calX(i_1,\ldots,i_{\mu-1}, i_{\mu+1}, \ldots, i_{\nu}, i_{\mu}, i_{\nu+1},\ldots, i_d )$.
The result follows from the observation that the permutation $(1,\ldots,\mu-1, \mu+1, \ldots, \nu, \mu, \nu+1,\ldots, d)$ has sign $(-1)^{|\mu - \nu|}$.
\end{proof}

Lemma~\ref{lemma:antimat} implies that the multilinear rank of $\calA$ always takes the form $(r,\ldots, r)$ for some $1 \le r \le n$. In the following, we will simply refer to $r$ as the multilinear rank of an antisymmetric tensor.

\subsection{Restrictions on the multilinear rank}

It is well known that skew-symmetric matrices have even rank. It turns out that this property does not extend to antisymmetric tensors; it is simple to construct tensors of higher order with odd multilinear ranks. However, the following theorem shows that antisymmetry still imposes some (weaker) restrictions on the ranks of antisymmetric tensors that are of small size $n$ relative to $d$.

\begin{theorem}\label{largestmultilinrank}
Let $\calA\in\R^{n\times n\times\cdots\times n}$ be an antisymmetric tensor of order $d \ge 3$. Then the multilinear rank $r$ of $\calA$ satisfies
\begin{itemize}
\item[(i)] $r=0$ for $n<d$;
\item[(ii)] $r\leq d$ for $n=d$ or $n=d+1$;
\item[(iii)] $r\leq n$ for $n\geq d+2$.
\end{itemize}
There exist tensors $\calA$ for which equality is attained in $(i)$--$(iii)$.
\end{theorem}

\begin{proof}
\begin{itemize}
\item[(i)] According to~\eqref{notzero}, all $d$ indices $i_1,\ldots, i_d \in [1, n]$
need to be mutually different for an entry $\calA(i_1,i_2,\ldots,i_d)$ to be nonzero. When $n<d$, this is clearly not possible and hence $\calA = 0$.

\item[(ii)] For $n=d$, the condition $r \leq d$ follows from the size of the matricizations. To show that equality is attained, consider the tensor
$\calA = d!\,\antisymm(\calX)$ where all entries of $\calX \in \R^{d\times \cdots \times d}$ are zero except for $\calX(1,2,\ldots,d) = 1$.
For arbitrary $1\le i \le d$ choose the permutation $p = (i,1,\ldots,i-1, i+1, \ldots, d)$.
By definition, $\calA(p) = \text{sign}(p) = (-1)^{i-1}$. By letting $j = j\big(p(2),\ldots,p(d)\big)$, it follows that the $j$th column of $\bfA_{(1)}$ equals $(-1)^{i-1} e_i$ with the $i$th unit vector $e_i \in \R^n$. In particular, $\bfA_{(1)}$ has $d$ linearly independent columns and is thus of rank $d$.

Now, let $n=d+1$ and assume, without loss of generality, that $\calA \not=0$. We denote the rows of the matricization
$\bfA_{(1)}$ by $\bfA_{1,(1)},\ldots,\bfA_{d+1,(1)} \in \R^{n^{d-1}}$.
This matricization has rank at most $d$ if we can show that these rows are linearly dependent. Let
\[\alpha_k := \calA(1,\ldots,k-1,k+1,\ldots,d+1), \quad k = 1,\ldots,d+1.\]
Since $\calA \not =0$, at least one $\alpha_k$ is different from zero. Let us now consider the column of $\bfA_{(1)}$ corresponding to a fiber $\calA(:,i_2,\ldots, i_d)$ for some $i_2,\ldots,i_d \in [1,d+1]$. We may assume that $i_2,\ldots,i_d$ are mutually distinct because otherwise this fiber is zero.
For the moment, we also assume that these indices are ordered, that is, $1 \le i_2 < i_3 < \cdots < i_d \le d+1$.
By the pigeon hole principle, there are two integers $1 \le k < \ell \le d+1$ such that $k,l \not\in \{i_2,\ldots, i_d\}$.
The situation is now as follows:
$$\begin{array}{ccccccccccc}
  i_2 & \cdots & i_k &  & i_{k+1} & \cdots & i_{l-1} &  & i_l & \cdots & i_d \\
  \veq & \cdots & \veq &  & \veq & \cdots & \veq &  & \veq & \cdots & \veq \\
  1 & \cdots & k-1 & k & k+1 & \cdots & l-1 & l & l+1 & \cdots & d+1
\end{array}$$
In particular, this implies
\begin{eqnarray*}
  \calA(k,i_2,\ldots, i_d) &=& (-1)^{k-1} \calA(i_2, \ldots, i_{k-1}, k, i_{k+1}, \ldots, i_d) = (-1)^{k-1} \alpha_\ell, \\
  \calA(\ell,i_2,\ldots, i_d) &=& (-1)^{\ell-2} \calA(i_2, \ldots, i_{\ell-1}, \ell, i_{\ell+1}, \ldots, i_d) = (-1)^{\ell-2} \alpha_k.
\end{eqnarray*}
Using that $\calA(i_1,i_2,\ldots, i_d)$ is only nonzero for mutually distinct indices, we arrive at the linear combination
\begin{eqnarray*}
  \sum_{i_1 = 1}^{d+1} (-1)^{i_1} \alpha_{i_1} \calA(i_1,i_2,\ldots, i_d) &=& (-1)^{k} \alpha_{k} \calA(k,i_2,\ldots, i_d)
 + (-1)^{\ell} \alpha_{\ell} \calA(\ell,i_2,\ldots, i_d) \\
 &=& (-1)^{2k-1} \alpha_{k} \alpha_{\ell} + (-1)^{2\ell-2} \alpha_{k} \alpha_{\ell} \\
 &=& -\alpha_{k} \alpha_{\ell}+\alpha_{k} \alpha_{\ell} = 0.
\end{eqnarray*}
Since this relation is not affected by a permutation of $i_2,\ldots, i_d$, it also holds if these indices are not ordered.
In summary, we have shown that
\[
 \sum_{i_1 = 1}^{d+1} (-1)^{i_1} \alpha_{i_1} \bfA_{i_1,(1)} = 0
\]
and thus the rank of $\bfA_{(1)}$ is at most $d$.

For $n=d+1$ equality is attained by the tensor used in the construction for $n = d$ bordered with zeros.

\item[(iii)] Let $n\geq d+2$. By the size of the matricization, $r\leq n$. To show that $r = n$ can be attained, let us first define the integer vector
$h = (1,2,\ldots,n,1,\ldots,d-1)$. We choose the tensor $\calX\in\R^{n\times n\times\cdots\times n}$ to be zero except for
\[
 \calX(h_k,h_{k+1},h_{k+2},\ldots,h_{k+d-1}) = -d!, \qquad k = 1,2,\ldots, n.
\]
The corresponding sets $\sigma_k = \{h_k,h_{k+1},h_{k+2},\ldots,h_{k+d-1}\} \subset \mathbb N$ all have cardinality $d$ for $k = 1,2,\ldots, n$.
The set $\{1, 3, 4,\ldots,d\} = \sigma_1 \setminus \{2\}$ is only contained in $\sigma_1$. In particular, $n> d \ge 3$ implies that it is not contained in
$\sigma_n = \{n,1,\ldots,{d-1\}}$.
This shows ${\calA(:,1,3,4,\ldots,d)} = e_2.$
Analogously, $\calA(:,2,4,5,\ldots,d+1) = e_3$ and $\calA(:,3,5,6,\ldots,{d+2)} = e_4$.
This construction can be continued until we arrive at the set $\{n-1,1,\ldots,d-2\} = \sigma_{n-1} \setminus \{n\}$, which is not contained in $\sigma_1$ because of $n-1>d$,
or in any other $\sigma_k$ except $\sigma_{n-1}$. Hence, $\calA(:,n-1,1,\ldots,d-2) = e_n$.
Finally, we have $\calA(:,n,2,\ldots,d-2) = e_1$. In summary, we have found $n$ linearly independent columns of $\bfA_{(1)}$ and, therefore, the multilinear rank of $\calA$ is $n$.
\end{itemize}
\end{proof}

\subsection{HOSVD}

Given a general tensor $\calX \in \R^{n_1 \times \cdots \times n_d}$, the higher-order singular value decomposition (HOSVD) introduced in~\cite{DeLathauwer2000}
proceeds by computing the SVDs of the matricizations $\mathbf{X}_{(\mu)}$, $1\leq \mu \leq d$, and letting $V_\mu \in \R^{n_\mu \times n_\mu}$ contain the left singular vectors. Setting $\calT = \calX \times_1 V^T_1\times_2 V^T_2\times_3\cdots\times_d V^T_d$ yields the Tucker decomposition
\[
 \calX = \calT \times_1 V_1\times_2 V_2 \cdots\times_d V_d.
\]
The truncated HOSVD for a given multilinear rank $(r_1,\ldots, r_d)$ with $r_\mu \le n_\mu$ is obtained by setting
\begin{equation}\label{eq:thosvd}
\calS \times_1 U_1\times_2 U_2 \cdots\times_d U_d,
\end{equation}
with $U_\mu = V_\mu(:,1:r_\mu)$ and $\calS = \calT(1:r_1, 1:r_2,\ldots, 1:r_\mu)$. This gives a quasi-best approximation of $\calX$, in the sense that the approximation error in the Frobenius norm, $\|\calX - \calS \times_1 U_1 \cdots\times_d U_d\|$, is within a factor $\sqrt{d}$ of the error of the best rank-$(r_1,\ldots, r_d)$ approximation. In particular, if $\calX$ happens to have multilinear rank $(r_1,\ldots, r_d)$ then the decomposition~\eqref{eq:thosvd} is exact.

We now apply the truncated HOSVD to obtain an approximation of multilinear rank $r$ to an antisymmetric tensor $\calA$. By Lemma~\ref{lemma:antimat}, all matrices $U_\mu$ in~\eqref{eq:thosvd} can be chosen equal to a fixed matrix $U$. In turn $\calS = \calA \times_1 U^T \cdots\times_d U^T$ is again antisymmetric. In summary, the truncated HOSVD described in Algorithm~\ref{alg:thosvd} automatically preserves structure and produces a quasi-best antisymmetric approximation.

\begin{algorithm}[H] \small
\caption{Truncated HOSVD of antisymmetric tensor}
\label{alg:thosvd}
\begin{algorithmic}
\STATE Compute matrix $U \in \R^{n \times r}$ containing the leading $r$ left singular vectors of $\mathbf{A}_{(1)}$.
\STATE Set $\calS=\calA \times_1 U^T \cdots\times_d U^T$.
\STATE Return approximation $\calS\times_1U \cdots \times_d U$.
\end{algorithmic}
\end{algorithm}

\begin{corollary}
Let $\calA$ be an antisymmetric tensor of order $d$. Then the multilinear rank $r$ of $\calA$ satisfies $r = 0$ or $r = d$ or $d+2 \le r \le n$.
Any of these ranks can be attained.
\end{corollary}

\begin{proof}
By the discussion above, an antisymmetric tensor of multilinear rank $r$ can be written as $\calA=\calS\times_1 U \cdots\times_d U$, where the $r\times \cdots \times r$ tensor $\calS$ is again antisymmetric and has multilinear rank $r$. The statement of the corollary now follows from applying Theorem~\ref{largestmultilinrank} to $\calS$.
\end{proof}

Let us inspect the case $r = d$ more closely. Any antisymmetric $d\times \cdots \times d$ tensor of order $d$ takes the form
\begin{equation} \label{eq:tensord}
\calS = \antisymm(\alpha e_1 \otimes e_2 \otimes \cdots \otimes e_d),
\end{equation}
for some $\alpha \in \R$; see also the construction in the proof of Theorem~\ref{largestmultilinrank}~$(i)$. By letting $U = [u_1, u_2,  \ldots, u_d]$, the truncated HOSVD implies that any antisymmetric tensor of order $d$ and multilinear rank $d$ takes the form
\[
 \calA = \antisymm(\alpha e_1 \otimes e_2 \otimes \cdots \otimes e_d) \times_1 U \times_2 U \cdots  \times_d U =
 \antisymm(\alpha u_1 \otimes u_2 \otimes \cdots \otimes u_d),
\]
verifying the claim~\eqref{eq:rank1} from the introduction.

\section{Low multilinear rank approximation}\label{sec:multilinrankapprox}

In this section, we discuss two iterative methods that aim to compute a best antisymmetric multilinear rank-$r$
approximation
\[
 \min\big\{ \|\calA - \calS \times_1 U \cdots \times_d U\|: \ \calS \in \R^{r\times \cdots \times r} \text{ antisymmetric}, U \in \R^{n\times r}  \big\},
\]
starting, for example, from the truncated HOSVD of $\calA$. Both methods are based on the fact that this minimization problem is equivalent to solving
\begin{equation} \label{eq:maxstiefel}
\max\big\{ \| \calA \times_1 U^T \cdots \times_d U^T\|:\ U \in \R^{n\times r} \text{ with } U^T U = I_r \big\}
\end{equation}
and setting $\calS = \calA \times_1 U^T \cdots \times_d U^T$; see, for example,~\cite{LathauwerLowrank2000}.

\begin{remark}
For symmetric tensors, there is numerical evidence (see, e.g.,~\cite{IshtevaJac2013}) that the best (unstructured) approximation of multilinear rank $r$ can usually be chosen symmetric. For $d = 2$ and general $r$, this follows from the spectral decomposition. For general $d$ and $r = 1$, this property has recently been shown by Friedland~\cite{Friedland2013}. For general $d$ and $r$, this question remains open.

For antisymmetric tensors, we will observe the analogous phenomenon below; it appears that the best (unstructured) approximation of multilinear rank $r$ can usually be chosen antisymmetric. For $d = 2$ and even $r$, this property follows from the real Schur decomposition. For $r = d$ and general $d$, we will see in Section~\ref{sec:rankd} that it is actually the unstructured rank-$1$ approximation that gives an antisymmetric multilinear rank-$d$ approximation.
\end{remark}

To simplify the presentation, we will consider the case $d = 3$ for the rest of this section; all developments extend in a relatively straightforward manner to general $d > 3$.

\subsection{HOOI}

The higher-order orthogonal iteration (HOOI) introduced in~\cite{SaadHOOI2006} is a popular approach to the best low multilinear rank approximation of a general tensor. It consists of applying alternating least squares (ALS) to the unstructured variant of the maximization problem~\eqref{eq:maxstiefel}:
\[
 \max\big\{ \| \calA \times_1 U_1^T \times_2 U_2^T \times_3 U_3^T\|:\ U_\mu \in \R^{n\times r} \text{ with } U_\mu^T U_\mu = I_r, \mu = 1,2,3 \big\}.
\]
One step of the method optimizes a single factor $U_\mu$ while keeping the other two factors fixed. The resulting optimization problem admits a straightforward solution by the SVD; see Algorithm~\ref{alg:HOOI}.

\begin{algorithm}[H] \small
\caption{HOOI for multilinear rank-$(r,r,r)$ approximation}
\label{alg:HOOI}
\begin{algorithmic}
\STATE Apply Algorithm~\ref{alg:thosvd} to choose initial factors $U_1 = U_2 = U_3 = U$.
\REPEAT
\STATE $\calX=\calA\times_2U_2^T\times_3U_3^T$
\STATE Compute matrix $U_1 \in \R^{n \times r}$ containing the leading $r$ left singular vectors of $\mathbf{X}_{(1)}$.
\STATE $\calY=\calA\times_1U_1^T\times_3U_3^T$
\STATE Compute matrix $U_2 \in \R^{n \times r}$ containing the leading $r$ left singular vectors of $\mathbf{Y}_{(2)}$.
\STATE $\calZ=\calA\times_1U_1^T\times_2U_2^T$
\STATE Compute matrix $U_3 \in \R^{n \times r}$ containing the leading $r$ left singular vectors of $\mathbf{Z}_{(3)}$.
\UNTIL convergence
\STATE $\calS=\calZ\times_3U_3^T$
\STATE Return approximation $\calS\times_1U_1\times_2U_2\times_3U_3$.
\end{algorithmic}
\end{algorithm}

Note that the iterates of Algorithm~\ref{alg:HOOI} are \emph{not} antisymmetric.
However, similarly as in the symmetric case, we have observed that in most of the cases Algorithm~\ref{alg:HOOI} converges towards an antisymmetric approximation; see Section~\ref{sec:approxnumexp} below. To antisymmetrize the output of Algorithm~\ref{alg:HOOI}, one could set all factors $U_\mu$ to be equal to one of them. Here we choose the factor $U_\mu$ for which~\eqref{eq:maxstiefel} gives the biggest value.

A simple antisymmetric variant of Algorithm~\ref{alg:HOOI} consists of setting \emph{all} factors to the factor that has been obtained from the SVD in one step. In the symmetric case, this variant has been observed to suffer from convergence problems~\cite{IshtevaJac2013} and we observed similar difficulties in the antisymmetric case.

\subsection{Jacobi algorithm}

In contrast to HOOI, the Jacobi algorithm proposed for symmetric tensors in~\cite{IshtevaJac2013} preserves structure, that is, all iterates stay symmetric. In this section, we develop a variant of this algorithm for antisymmetric tensors.

It will be convenient to rewrite the maximization problem~\eqref{eq:maxstiefel} as
\[
 \max\big\{ f(Q):\ Q \in \R^{n\times n} \text{ with } Q^T Q = I_n \big\},
\]
where
\begin{equation} \label{eq:deff}
f(Q) = \|\calA\times_1{M}{Q}^T\times_2{M}{Q}^T\times_3{M}{Q}^T\|^2, \qquad
M =  \begin{pmatrix}
I_r & 0 \\
0 & 0 \\
\end{pmatrix}.
\end{equation}
We will denote a Givens rotation acting on rows/columns $i$ and $j$ by
\[
 R(i,j,\phi) = \begin{gmatrix}[b]
 I &  &  &  & \\
  & \cos\phi &  & -\sin\phi & \\
  &  & I &  &    \\
  & \sin\phi &  & \cos\phi& \\
  &  &  &  & I
 \colops
  \mult{1}{i}
  \mult{3}{j}
 \rowops
  \mult{1}{i}
  \mult{3}{j}
\end{gmatrix}.
\]
In the following, $(i,j)$ will be called a pivot pair.

The main idea of the Jacobi algorithm is to repeatedly apply Givens rotations that increase the norm of the $(1:r,1:r,1:r)$ subtensor.
For this purpose, it will be sufficient to consider rotations corresponding to the pivot pairs
\begin{equation}\label{pivotpairs}
\begin{array}{ccccc}
(1,r+1), & (1,r+2), & \ldots & (1,n) \\
(2,r+1), & (2,r+2), & \ldots & (2,n) \\
\vdots & \vdots & & \vdots \\
(r,r+1), & (r,r+2), & \ldots & (r,n).
\end{array}
\end{equation}

In every iteration of the Jacobi algorithm, we choose a pivot pair that produces a direction of sufficiently strong descent.
Letting
$$d_{ij}= \frac{\partial}{\partial \phi} R(i,j,\phi) \Big|_{\phi=0}=\begin{gmatrix}[b]
 I &  &  &  & \\
  & 0 &  & -1 & \\
  &  & I &  &    \\
  & 1 &  & 0& \\
  &  &  &  & I
 \colops
  \mult{1}{i}
  \mult{3}{j}
 \rowops
  \mult{1}{i}
  \mult{3}{j}
\end{gmatrix},$$
we can always find pivot pairs $(i,j)$ among~\eqref{pivotpairs} such that
\begin{equation}\label{Jac condition}
|\langle\text{grad}f(I),d_{ij}\rangle|\geq\epsilon\|\text{grad}f(I)\|
\end{equation}
holds, provided that $0 < \epsilon < 2/n$; see~\cite[Lemma 5.2]{IshtevaJac2013}.

Once a pivot pair $(i,j)$ satisfying~\eqref{Jac condition} is determined, we choose the
rotation angle $\phi$ that maximizes $f$, i.e., we solve
\begin{equation} \label{eq:maxphi}
\max\big\{ f\big(R(i,j,\phi) \big):\ \phi \in [0,\pi] \big\},
\end{equation}
Because of $1\leq i\leq r<j\leq n$, the tensor $\calB = \calA \times_1 R(i,j,\phi)^T \times_2 R(i,j,\phi)^T \times_3 R(i,j,\phi)^T$ differs from $\calA$ within the subtensor $(1:r,1:r,1:r)$ only in the three slices $(i,1:r,1:r)$, $(1:r,i,1:r)$, and $(1:r,1:r,i)$. Because of antisymmetry, these slices have identical norms and we can ignore their intersections. Hence,~\eqref{eq:maxphi} becomes equivalent to maximizing
\begin{eqnarray}
\|\calB(i,1:r,1:r)\|^2 & = & \sum_{p,q = 1}^r \calB(i,p,q)^2 = \sum_{\substack{p,q=1\\p,q\neq i}}^r \calB(i,p,q)^2 \nonumber \\
&=& \sum_{\substack{p,q=1\\p,q\neq i}}^r \big( \cos \phi \calA(i,p,q)+\sin\phi \calA(j,p,q) \big)^2 =:\psi(\phi). \label{psi}
\end{eqnarray}
Let
$$\alpha_1=\sum_{\substack{p,q=1\\p,q\neq i}}^r\calA(i,p,q)^2, \ \alpha_2=\sum_{\substack{p,q=1\\p,q\neq i}}^r\calA(i,p,q)\calA(j,p,q),
\ \alpha_3=\sum_{\substack{p,q=1\\p,q\neq i}}^r\calA(j,p,q)^2.$$
Then the derivative of~\eqref{psi} takes the form
$$\psi^\prime(\phi) = -2 \alpha_1 \cos\phi \sin\phi +2\alpha_2 (\cos^2 \phi-\sin^2 \phi)+2\alpha_3\cos\phi\sin\phi.$$
In order to find the zeros of this function, we divide it by $\cos^2 \phi$
and solve the resulting quadratic equation in $t=\sin\phi / \cos\phi$:
$$\alpha_2 t^2 +(\alpha_1-\alpha_3)t-\alpha_2=0.$$
Among the two solutions to this equation, we choose the one that maximizes~\eqref{psi}.
Algorithm~\ref{alg:jacobi} summarizes the described procedure.

\begin{algorithm}[H] \small
\caption{Jacobi algorithm for antisymmetric multilinear rank-$r$ approximation}
\label{alg:jacobi}
\begin{algorithmic}
\STATE Apply Algorithm~\ref{alg:thosvd} to choose initial factor $U \in \R^{n\times r}$.
\STATE Choose $U_\perp$ such that $Q = [U,U_\perp]$ is orthogonal.
\STATE Set $\calA_1 = \calA \times_1 Q^T \times_2 Q^T \times_3 Q^T$.
\REPEAT
\STATE {Choose $(i,j)$ according to~\eqref{pivotpairs} and~\eqref{Jac condition}.}
\STATE {Determine $\phi$ that maximizes~\eqref{psi}.}
\STATE $Q_{k+1}=Q_kR(i,j,\phi)$
\STATE $\calA_{k+1}=\calA_k\times_1R(i,j,\phi)^T\times_2R(i,j,\phi)^T\times_3R(i,j,\phi)^T$
\UNTIL {convergence}
\STATE $U=Q_k(:,1:r)$
\STATE Return approximation $\calA \times_1UU^T \times_2 UU^T \times_3 UU^T$.
\end{algorithmic}
\end{algorithm}

For choosing the pivot pair $(i,j)$ in Algorithm~\ref{alg:jacobi}, we traverse the list~\eqref{pivotpairs} cyclically. For each pair, the condition~\eqref{Jac condition} is checked. If $(i,j)$ does not fulfill this condition, it is skipped and the algorithm continues checking the next pair.

Although observed in practice, it cannot be guaranteed that Algorithm~\ref{alg:jacobi} produces the minimum of the function in~\eqref{eq:deff}.
The proof of a weaker convergence result for symmetric tensors~\cite[Theorem 5.4]{IshtevaJac2013} directly extends to antisymmetric tensors, resulting in Theorem~\ref{Jacobi cvgtm}.

\begin{theorem}\label{Jacobi cvgtm}
Let $(Q_k)$ be the sequence of orthogonal matrices generated by Algorithm~\ref{alg:jacobi} applied to an antisymmetric tensor $\calA\in\R^{n\times n\times n}$.
Then every accumulation point of $(Q_k)$ is a stationary point of the function $f$ from~\eqref{eq:deff}.
\end{theorem}

\subsection{Numerical Experiments}\label{sec:approxnumexp}

The algorithms described in this paper have been implemented and tested in Matlab version 7.11.

In our first set of experiments, we study the approximation error obtained by truncated HOSVD, HOOI, and the Jacobi algorithm. The latter two algorithms are iterative; they are considered converged when the norm of the gradient of the objective function is $10^{-10}$ or below. We have chosen $\epsilon = 1/(10 n)$ in the condition~\eqref{Jac condition} of the Jacobi algorithm.
We tested the algorithms with random tensors generated by applying antisymmetrizer from~\eqref{eq:antisymmetrizer} to tensors with uniformly distributed random entries from the interval $[0,1]$.
Figure~\ref{fig:mlr10} shows that HOOI and the Jacobi algorithm always improve upon the approximation obtained from the HOSVD.
In many cases, HOOI and the Jacobi algorithm result in the same (antisymmetric) approximation. In rare cases when the error of the Jacobi algorithm is greater than the one of HOOI, it is observed that the tensor produced by HOOI is not antisymmetric. On the other hand, when the error of HOOI is greater, the tensor produced by HOOI is antisymmetric. Although providing little evidence, these observations do at least not contradict a conjecture that the best (unstructured) approximation of multilinear rank $(r,r,r)$
to a generic antisymmetric tensor can always be chosen antisymmetric for $r\ge 3$.

\begin{figure}[h]
    \centering
    \begin{subfigure}[b]{0.45\textwidth}
        \includegraphics[width=\textwidth]{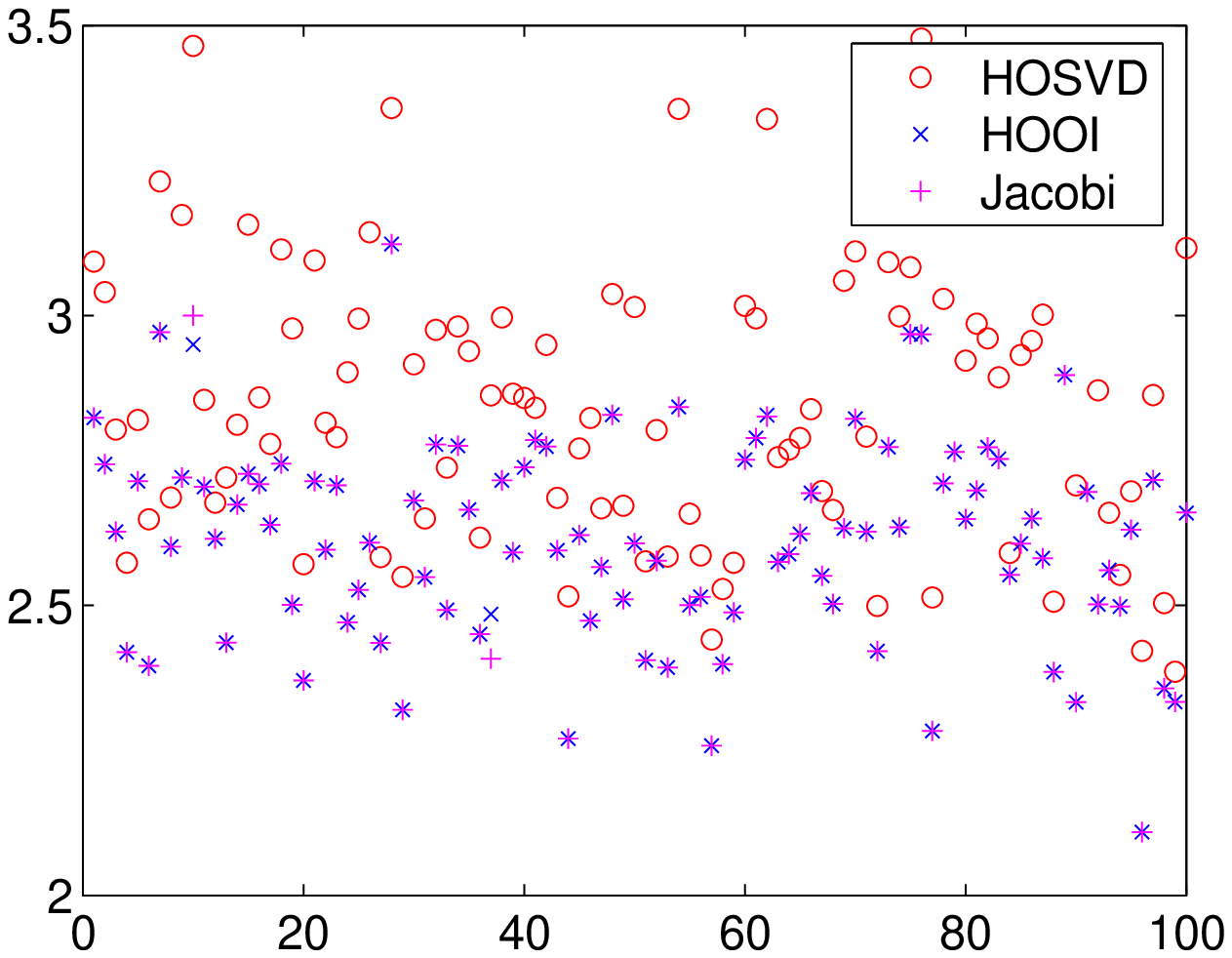}
        \caption{Multilinear rank $3$}
    \end{subfigure}
    \begin{subfigure}[b]{0.45\textwidth}
        \includegraphics[width=\textwidth]{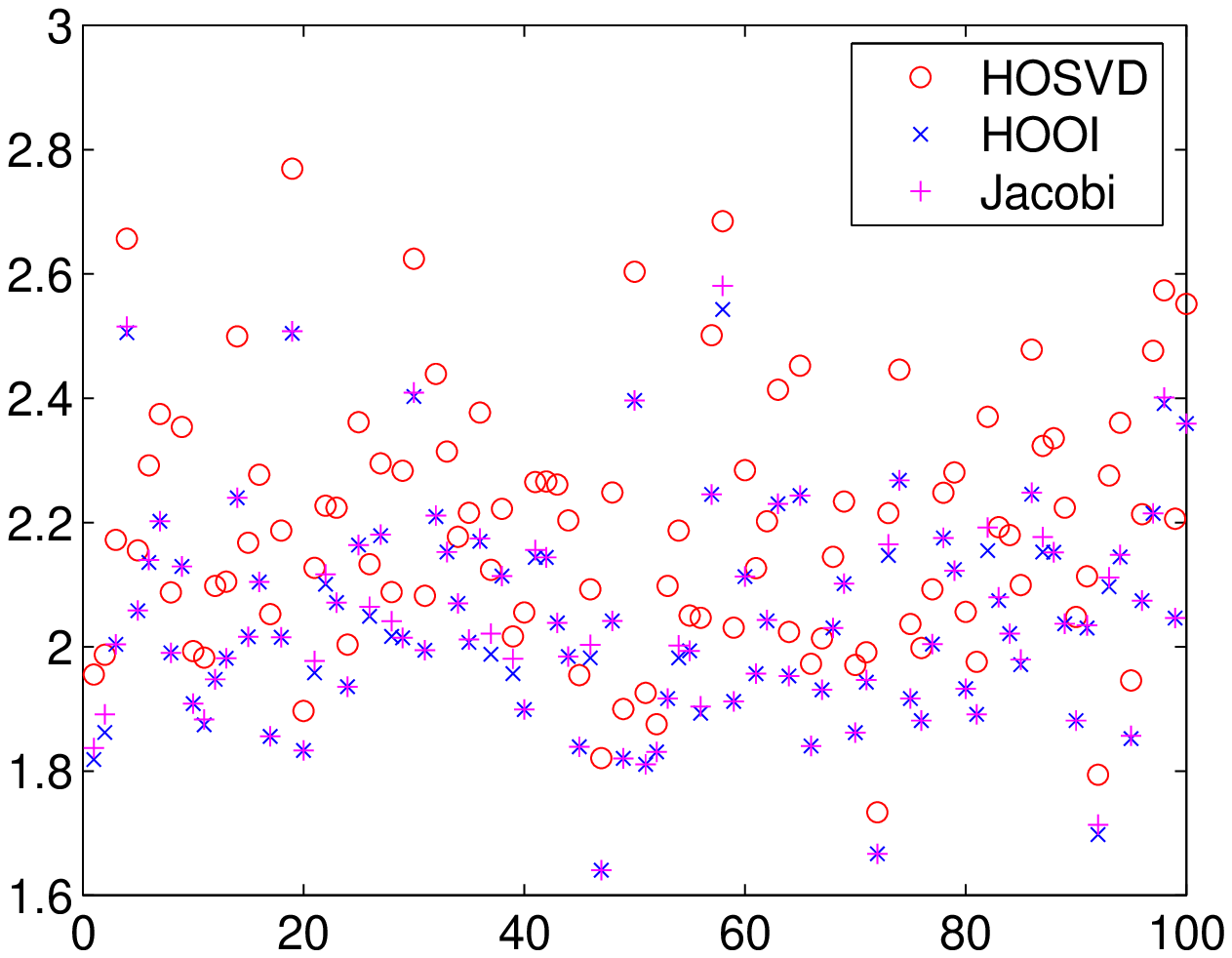}
        \caption{Multilinear rank $6$}
    \end{subfigure}
    \caption{Approximation error of low multilinear rank approximation to $100$ random antisymmetric $10\times10\times10$ tensors.}\label{fig:mlr10}
\end{figure}

Figure~\ref{fig:mlgrad} yields insights into the convergence behavior of HOOI and Jacobi algorithm for a representative run with a random antisymmetric tensor. To emphasize the benefits from initializing with the truncated HOSVD we compare with using no initialization, that is, instead of using Algorithm~\ref{alg:thosvd} we set $U = \Big[{I_r \atop 0}\Big]$
and $Q = I_n$ in Algorithms~\ref{alg:HOOI} and~\ref{alg:jacobi}, respectively. Apart from the approximation error we also show the norm of the gradient of the objective function.

\begin{figure}[h]
    \centering
    \begin{subfigure}[b]{0.45\textwidth}
        \includegraphics[width=\textwidth]{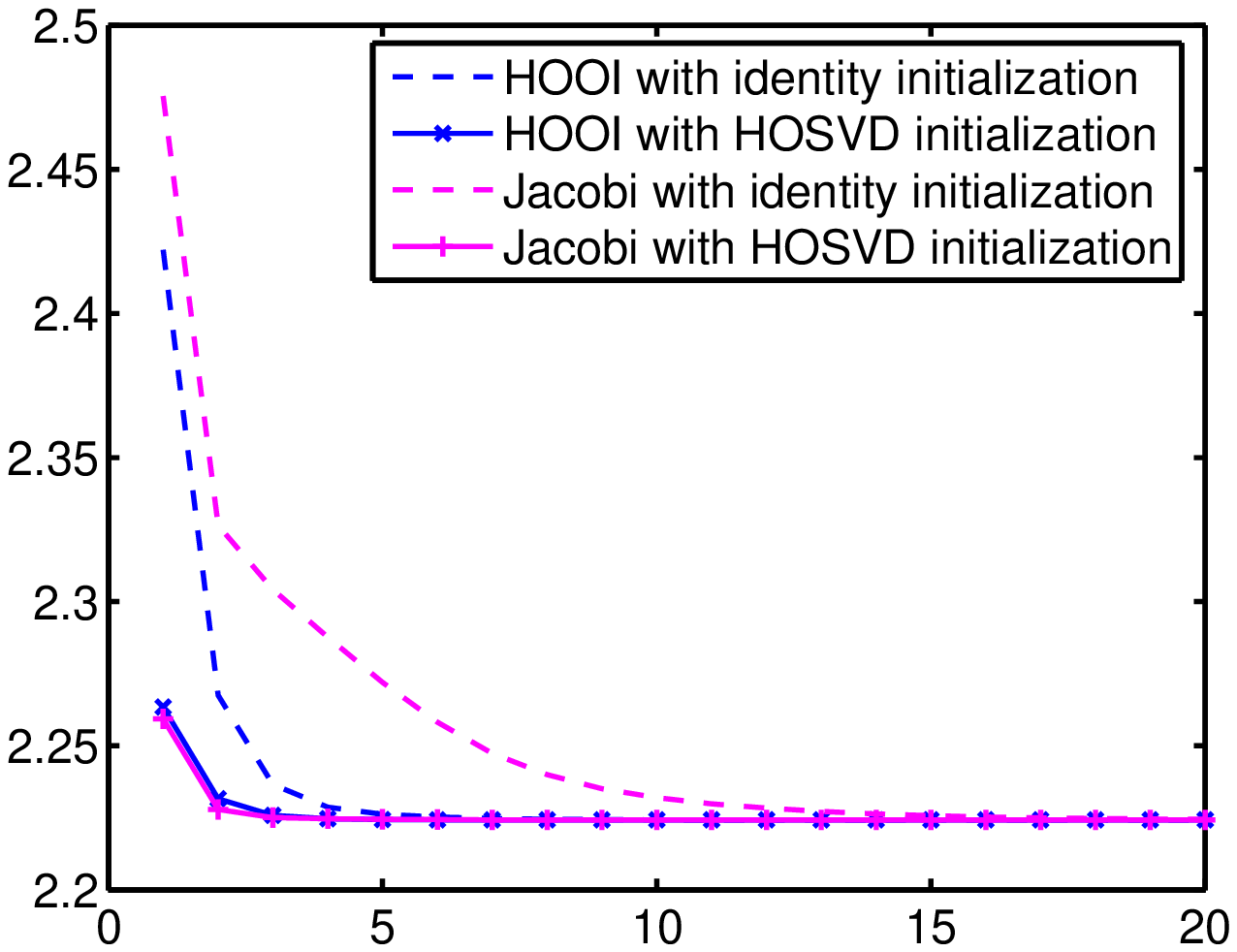}
        \caption{Approximation error}
    \end{subfigure}
    \begin{subfigure}[b]{0.45\textwidth}
        \includegraphics[width=\textwidth]{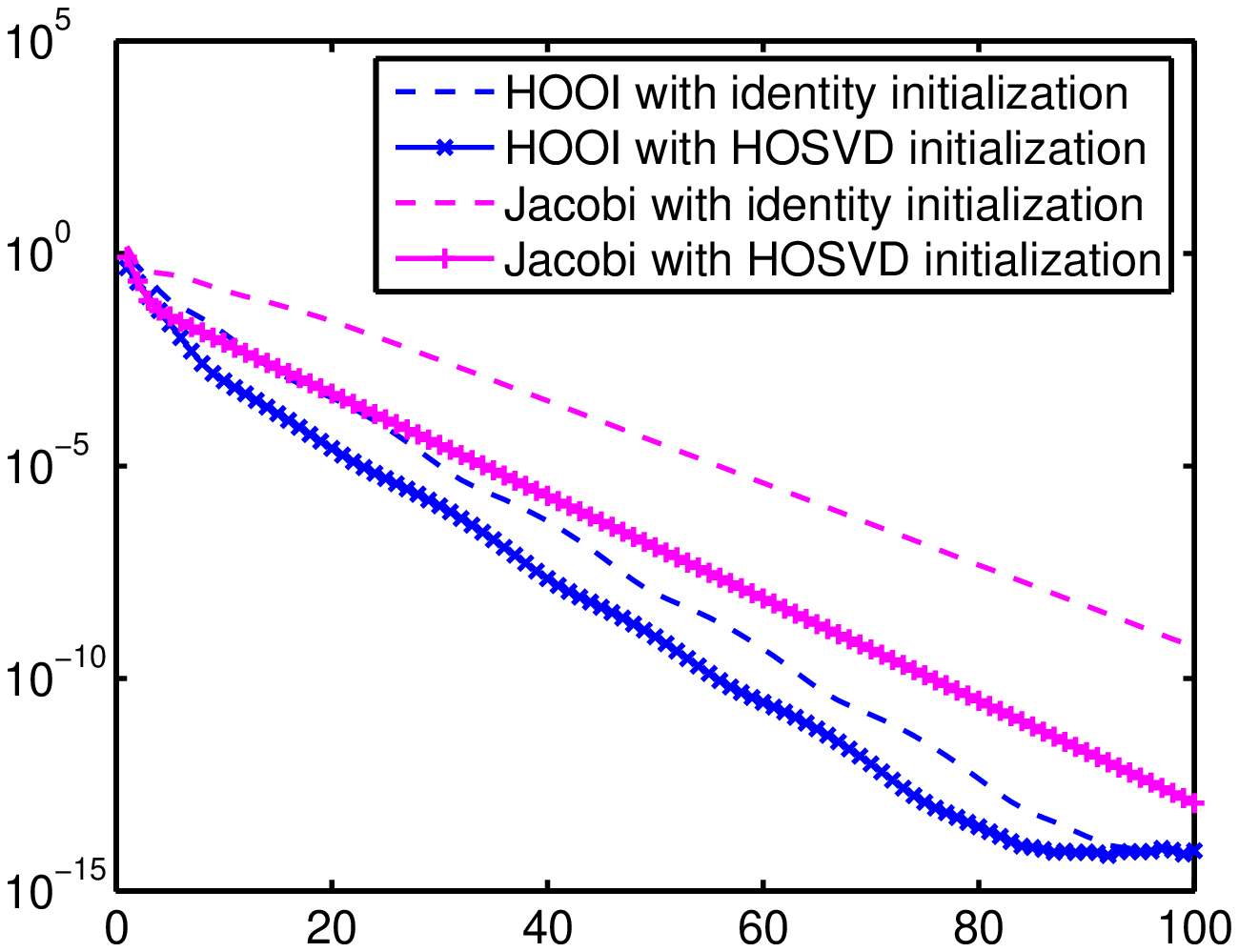}
        \caption{Norm of gradient of objective function}
    \end{subfigure}
    \caption{Convergence behavior of HOOI and Jacobi algorithm for multilinear rank-$6$ approximation of random antisymmetric $10\times10\times10$ tensor.}\label{fig:mlgrad}
\end{figure}

We have also considered antisymmetric tensors for which the matricizations are known to exhibit rapid singular value decays. To construct such a tensor, consider the function
\[ f(x,y,z) = \exp( -\sqrt{x^2 + 2 y^2 + 3z^2})\]
on $[0,1]^3$. Then we let $\calX$ contain its discretization:
\[
 \calX(i_1,i_2,i_3) = f\big(\xi_{i_1}, \xi_{i_2}, \xi_{i_3}\big), \quad i_\mu = 1,\ldots, n,
\]
where $\xi_i = (i-1)/(n-1)$, and set $\calA = \antisymm(\calX)$.
Figure~\ref{fig:Fcvg} shows the obtained results for $n = 20$. It reveals that the HOSVD gives an excellent initial approximation. This is also an example where the Jacobi algorithm with no initialization fails to converge to a global optimum.

\begin{figure}[h]
    \centering
    \begin{subfigure}[b]{0.45\textwidth}
        \includegraphics[width=\textwidth]{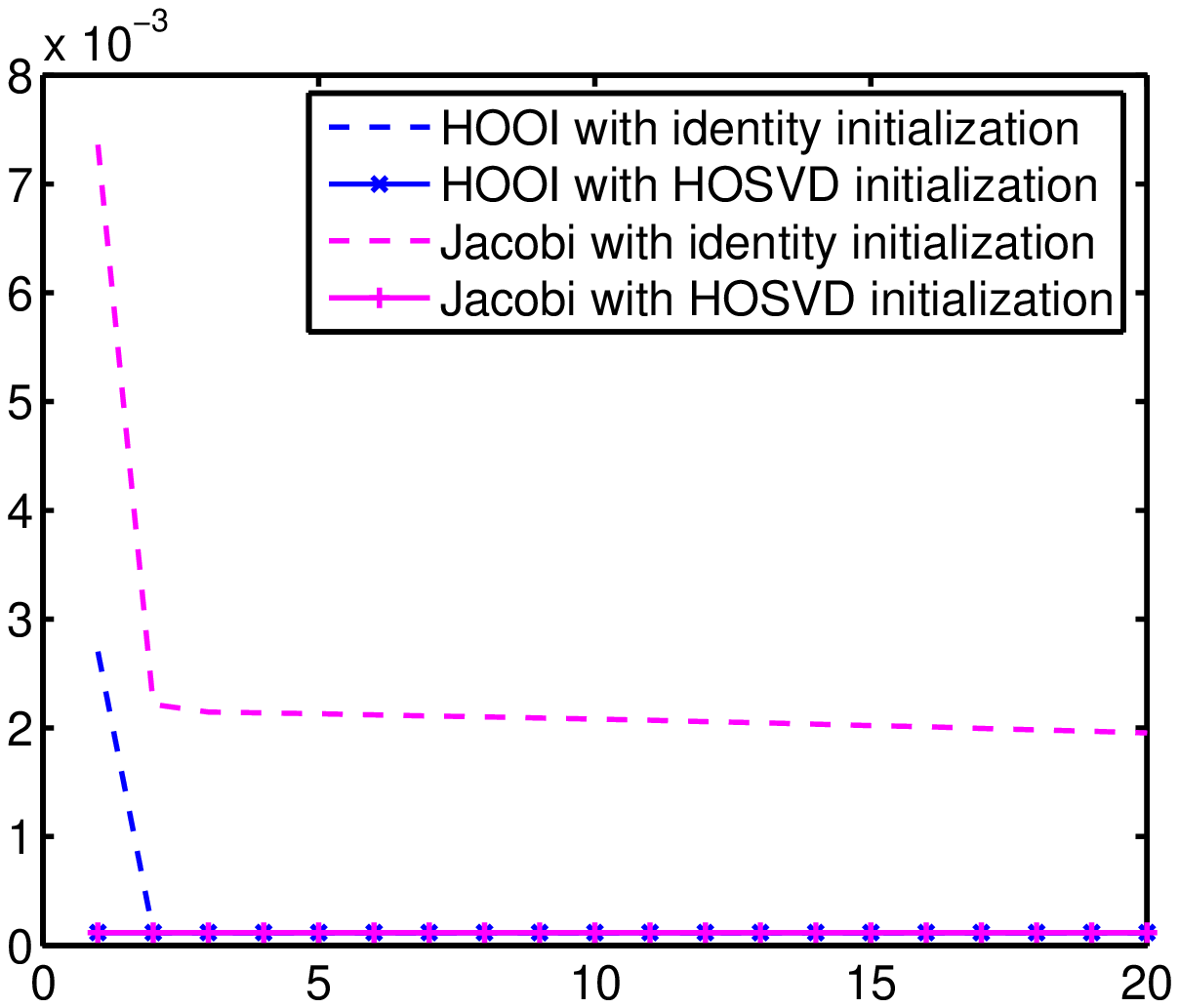}
        \caption{Approximation error}
    \end{subfigure}
    \begin{subfigure}[b]{0.45\textwidth}
        \includegraphics[width=\textwidth]{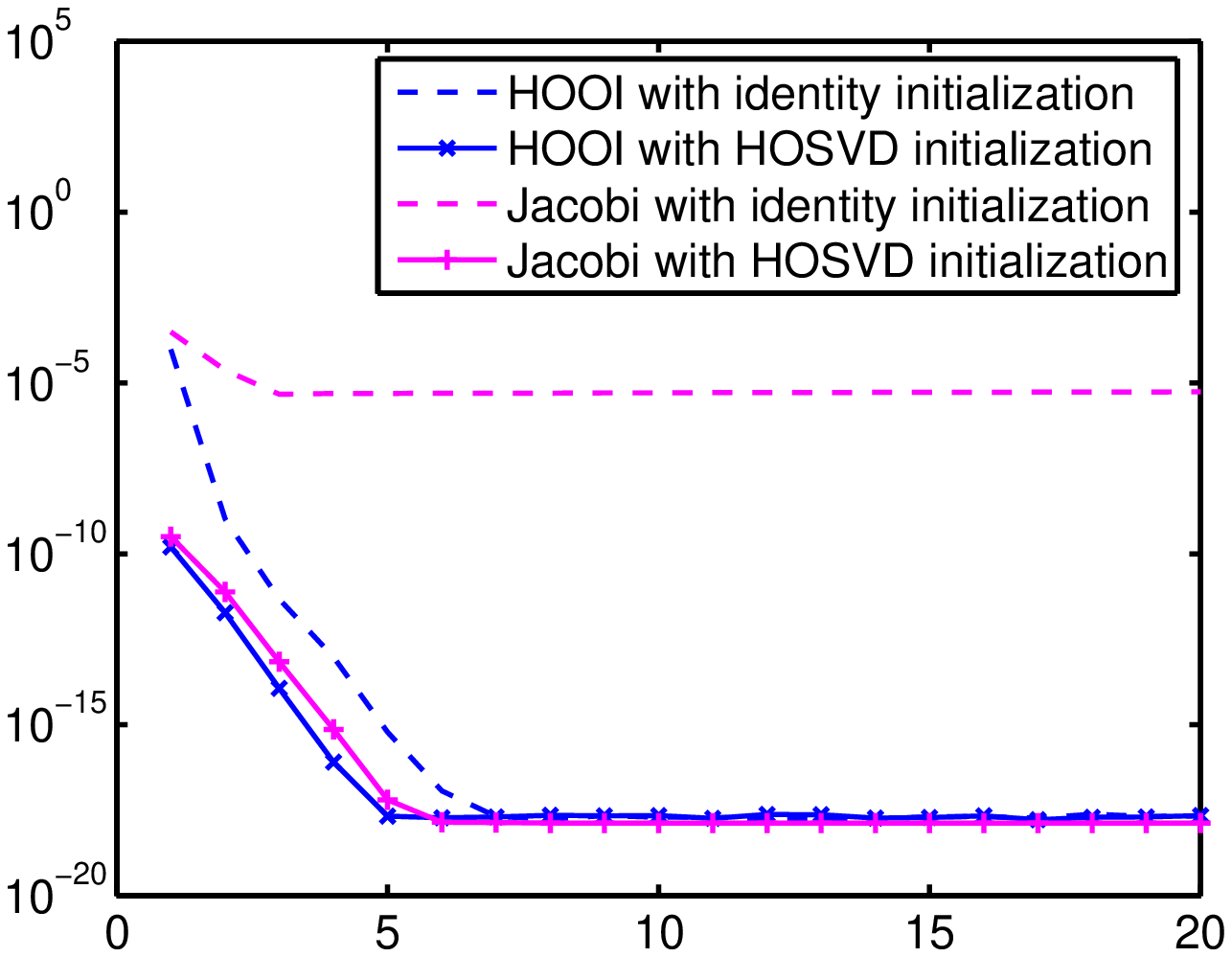}
        \caption{Norm of gradient of objective function}
    \end{subfigure}
    \caption{Convergence behavior of HOOI and Jacobi algorithm for multilinear rank-$7$ approximation of function-related tensor.}\label{fig:Fcvg}
\end{figure}

Finally, we report results for the antisymmetric ground state of a Schr\"odinger eigenvalue problem, which motivated our study of antisymmetric tensors. Following the example from~\cite[Sec. 5.4]{Beylkin2005}, we consider one-dimensional $2\pi$-periodic variables $x_1,\ldots,x_d \in \R$ and the Hamiltonian
\begin{equation} \label{eq:hamiltonian}
\mathbb H = -\frac12 \sum_{\mu = 1}^d \frac{\partial^2}{\partial x_\mu^2} + c_v \sum_{\mu = 1}^d \cos(2\pi x_\mu) + c_w \sum_{\mu = 1}^{d-1} \sum_{\nu = \mu+1}^{d} \cos(2\pi(x_\mu-x_\nu)),
\end{equation}
with $c_v = 100$, $c_w = 5$. The goal is to compute the smallest (negative) eigenvalue of $\mathbb H$ with an antisymmetric eigenfunction. After discretizing each variable $x_\mu$ on a uniform grid with $n$ grid points in $[0,2\pi)$ and approximating each $\partial^2/ \partial x_\mu^2$ with central finite differences, $\mathbb H$ becomes an $n^d \times n^d$ matrix, which can be reinterpreted as a linear operator $\mathcal H: \R^{n\times \cdots \times n} \to \R^{n\times \cdots \times n}$. Because $\mathbb H$ is invariant under permutations of the variables, $\mathcal H$ commutes with the antisymmetrizer $\mathcal A$. Hence, computing the most negative eigenvalue of $\mathcal H$ with an antisymmetric eigenvector is equivalent to computing the smallest eigenvalue of $\mathcal A \circ \mathcal H$. We used {\sc Matlab}'s {\tt eigs} for the latter and then applied the HOOI and Jacobi algorithm to the resulting (antisymmetric) eigenvector. The obtained results for $d = 3$ and $n = 50$ shown in Figure~\ref{fig:beylkin} are qualitatively similar to the ones obtained for the function-related tensor above.

\begin{figure}[h]
    \centering
    \begin{subfigure}[b]{0.45\textwidth}
        \includegraphics[width=\textwidth]{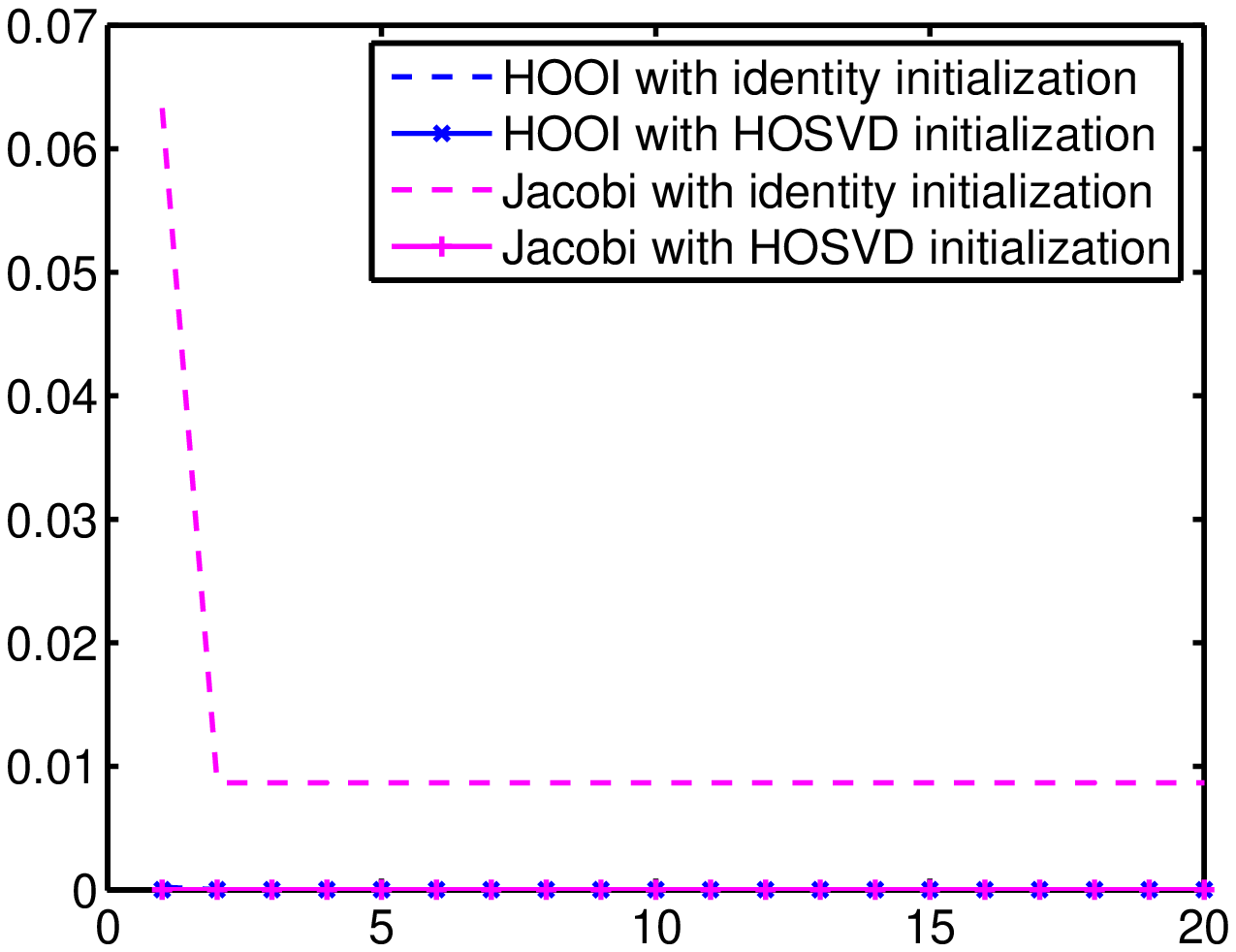}
        \caption{Approximation error}
    \end{subfigure}
    \begin{subfigure}[b]{0.45\textwidth}
        \includegraphics[width=\textwidth]{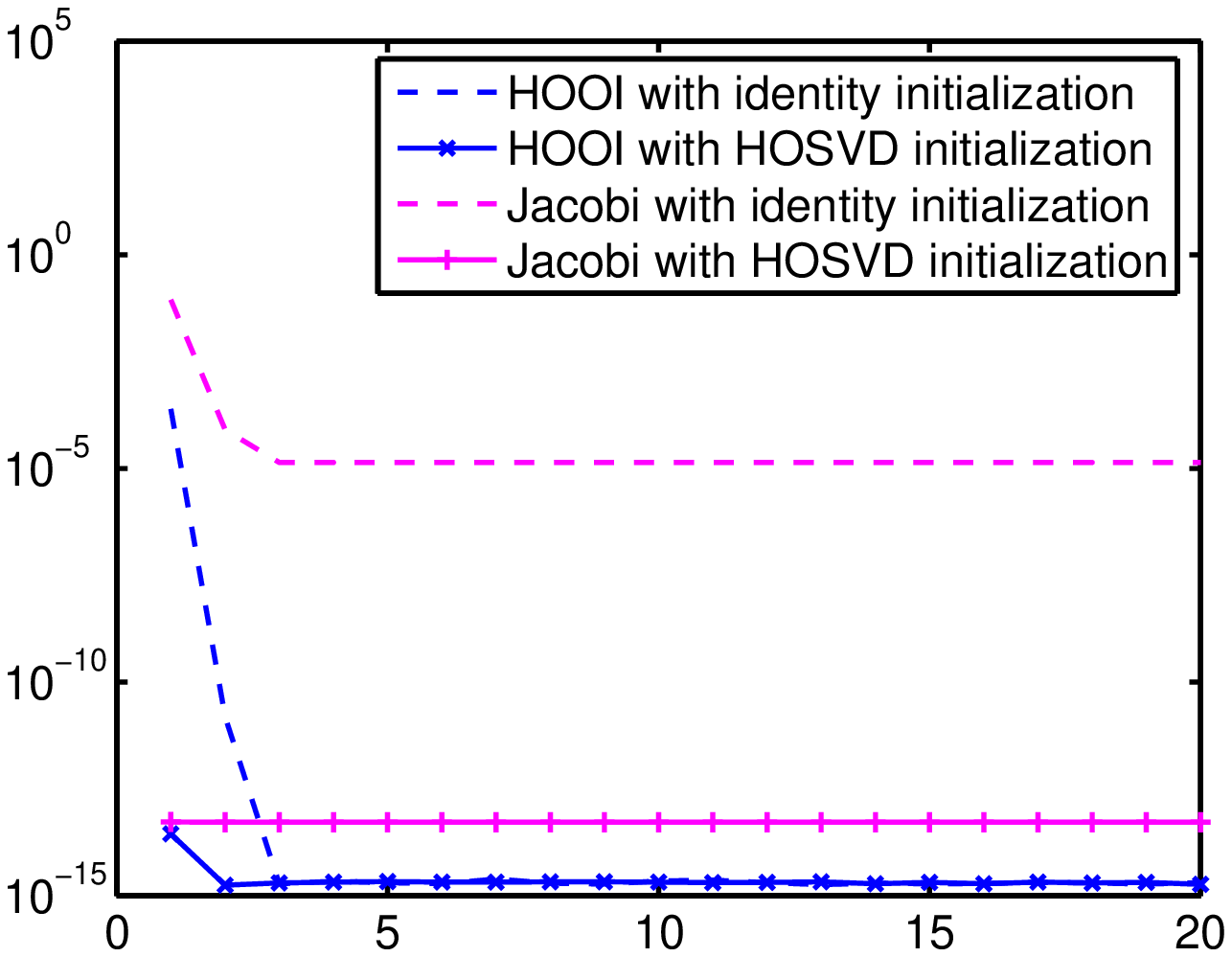}
        \caption{Norm of gradient of objective function}
    \end{subfigure}
    \caption{Convergence behavior of HOOI and Jacobi algorithm for multilinear rank-$7$ approximation of antisymmetric ground state for the discretized Hamiltonian~\eqref{eq:hamiltonian} with $d = 3$ and $n = 50$.}\label{fig:beylkin}
\end{figure}

\section{Multilinear rank-$d$ approximation}\label{sec:rankd}

Antisymmetric tensors of order $d$ and multilinear rank $d$ have the very particular structure~\eqref{eq:rank1}. As we will discuss in this section, this simplifies the approximation with such tensors significantly.

The following basic lemma plays a key role; it extends the well known fact that $u^T A u = 0$ always holds for a skew-symmetric matrix $A$.

\begin{lemma}\label{lemma:skewscalar}
Let $\calA\in\R^{n\times \cdots\times n}$ be an antisymmetric tensor of order $d \ge 2$ and $u\in \R^n$.
Then $\calA \times_\mu u \times_\nu u= 0$ for any $1 \le \mu < \nu \le d$.
\end{lemma}

\begin{proof}
Without loss of generality, we may assume that $\mu  = d-1$ and $\nu = d$. Then any entry of $\calB = \calA \times_\mu u \times_\nu u$ satisfies
\begin{eqnarray*}
\calB(i_1,\ldots, i_{d-2}) &=& \sum_{j, k = 1}^n \calA(i_1, \ldots, i_{d-2},j,k) u_j u_k \\
&=& -\sum_{j, k = 1}^n \calA(i_1, \ldots, i_{d-2},k,j) u_j u_k = -\calB(i_1,\ldots, i_{d-2}),
\end{eqnarray*}
which implies $\calB = 0$.
\end{proof}

The following theorem establishes an equivalence between the best antisymmetric multilinear rank-$d$ approximation and the best unstructured rank-$1$ approximation of an antisymmetric tensor.

\begin{theorem}\label{thm:antirank1}
Let $\calA \in \R^{n\times \cdots \times n}$ be an antisymmetric tensor of order $d$. Then
\begin{eqnarray}
&& \max\big\{ \| \calA \times_1 U^T \cdots \times_d U^T\|:\ U \in \R^{n\times d} \text{ \rm  with } U^T U = I_d \big\} \nonumber \\
&=& d! \max\big\{| \calA \times_1 u_1^T \cdots \times_d u_d^T|:\  [u_1,\ldots,u_d]^T [u_1,\ldots,u_d] = I_d \big\} \label{eq:firstequality} \\
&=& d! \max\big\{| \calA \times_1 v_1^T \cdots \times_d v_d^T|:\  \|v_1\| = \cdots = \|v_d\| = 1 \big\}. \label{eq:secondequality}
\end{eqnarray}
\end{theorem}

\begin{proof}
Let $\alpha = | \calA \times_1 u_1^T \cdots \times_d u_d^T|$. Using~\eqref{eq:tensord},
\[
  \| \calA \times_1 U^T \cdots \times_d U^T\|^2 = \| \antisymm(\alpha e_1 \otimes \cdots \otimes e_d) \|^2 = \alpha^2 \| \antisymm(e_1 \otimes \cdots \otimes e_d) \|^2 = (\alpha d!)^2,
\]
which shows~\eqref{eq:firstequality}.

Consider vectors $v_1,\ldots,v_d$ assuming the maximum in~\eqref{eq:secondequality}. By the QR decomposition, there is an upper triangular matrix $R \in \R^{d\times d}$ with $|r_{\mu\mu}| \le 1$  for $\mu = 1,\ldots,d$ and a matrix $[u_1,\ldots,u_d]$ with orthonormal columns such that
\[
 [v_1,\ldots,v_d] = [u_1,\ldots,u_d] R.
\]
Using Lemma~\ref{lemma:skewscalar},
\begin{eqnarray*}
\calA \times_1 v_1^T \cdots \times_d v_d^T &=& \calA \times_1 \sum_{\mu_1 = 1}^d r_{\mu_1,1} u_{\mu_1}^T \cdots \times_d \sum_{\mu_d = d}^d r_{\mu_d,d} u_{\mu_d}^T  \\
&=& \sum_{\mu_1 = 1}^d \cdots \sum_{\mu_d = d}^d \calA \times_1 r_{\mu_1,1} u_{\mu_1}^T \cdots \times_d r_{\mu_d,d} u_{\mu_d}^T \\
&=& \calA \times_1 r_{11} u_{1}^T \cdots \times_d r_{dd} u_{d}^T = r_{11} \cdots r_{dd} \calA \times_1 u_{1}^T \cdots \times_d u_{d}^T.
\end{eqnarray*}
Because of $|r_{\mu\mu}| \le 1$, this implies $|\calA \times_1 v_1^T \cdots \times_d v_d^T| \le |\calA \times_1 u_{1}^T \cdots \times_d u_{d}^T|$ and hence~\eqref{eq:secondequality} cannot be larger than~\eqref{eq:firstequality}.
On the other hand, trivially,~\eqref{eq:firstequality} cannot be larger than~\eqref{eq:secondequality}. This shows the equality~\eqref{eq:secondequality}.
\end{proof}

\subsection{HOPM}

Algorithm~\ref{alg:HOPM} recalls the higher-order power method (HOPM) proposed in~\cite{LathauwerLowrank2000} for finding a rank-$1$ approximation of a tensor $\calA \in \R^{n\times \cdots \times n}$.

\begin{algorithm}[H] \small
\caption{HOPM for rank-$1$ approximation}\label{alg:HOPM}
\begin{algorithmic}
\STATE Choose initial vectors $u_1, \ldots, u_d \in \R^n$. 
\REPEAT
\STATE $v_1 =\calA\times_2 u_2^T\times_3u_3^T \cdots \times_d u_d^T$
\STATE $u_1 = v_1 / \|v_1\|$
\STATE $v_2 =\calA\times_1 u_1^T\times_3u_3^T \cdots \times_d u_d^T$
\STATE $u_2 = v_2 / \|v_2\|$
\STATE $\ \ \!\quad \vdots$
\STATE $v_d =\calA\times_1 u_1^T\times_2u_2^T \cdots \times_{d-1} u_{d-1}^T$
\STATE $u_d = v_d / \|v_d\|$
\UNTIL convergence
\STATE $\alpha=\calA \times_1 u_1^T \cdots \times_d u_d^T$
\STATE Return approximation $\alpha u_1 \otimes u_2 \otimes \cdots \otimes u_d$.
\end{algorithmic}
\end{algorithm}

Assuming that Algorithm~\ref{alg:HOPM} converges to a best rank-$1$ approximation with mutually orthogonal vectors $u_\mu$, Theorem~\ref{thm:antirank1} allows us to construct the best antisymmetric multilinear rank-$d$ approximation
$\antisymm( \alpha u_1 \otimes \cdots \otimes u_d ).$
The following lemma assures mutual orthogonality.

\begin{lemma}
Let $\calA \in \R^{n\times \cdots \times n}$ be an antisymmetric tensor of order $d\le n$.
Then the vectors $u_1,\ldots, u_d$ returned by HOPM form an orthonormal basis, provided that HOPM does not encounter a zero vector.
\end{lemma}

\begin{proof}
After the first step of HOPM, Lemma~\ref{lemma:skewscalar} implies
\[
 \langle v_1, u_\nu \rangle = \calA\times_1 u_\nu^T \times_2 u_2^T\times_3u_3^T \cdots \times_d u_d^T = 0
\]
for any $\nu \not=1$. Hence, each step of HOPM orthogonalizes one of the vectors $u_\mu = v_\mu / \|v_\mu\|$.
In turn, the statement of the lemma holds after at least one sweep of HOPM, even if the initial vectors are not orthogonal.
\end{proof}

\begin{remark}
To ensure orthogonality numerically, we perform another orthogonalization step after each step of Algorithm~\ref{alg:HOPM}.
In principle, this procedure can also be applied to HOOI, yielding an (unstructured) multilinear rank $(r_1,\ldots,r_d)$ approximation with mutually orthogonal basis matrices $U_\mu$. This can then be turned into an antisymmetric multilinear rank-$(r_1 + \cdots + r_d)$ approximation by setting $U = [U_1,\ldots,U_d]$. We have tested this idea numerically and observed that this often yields a good approximation but the approximation error is usually worse compared to the result of the Jacobi algorithm.
\end{remark}

\subsection{Initialization}

It remains to discuss a proper initialization strategy for HOPM. For general $d$, we use the truncated HOSVD from Section~\ref{sec:multilinrankapprox}. For $d = 4$, we propose an antisymmetric variant of the technique proposed by Kofidis and Regalia~\cite{Kofidis2002} for symmetric tensors. Antisymmetric tensors of order $d = 4$ appear, for example, for wave functions describing the position of four fermions.
For this purpose, we define the $(1,2)$-matricization of a $4$th-order tensor $\calX \in \R^{n_1\times n_2\times n_3 \times n_4}$ to be the $n_1 n_2 \times n_3 n_4$ matrix $\mathbf{X}_{(1,2)}$ with the entries
\begin{equation}\label{eq:mat12}
\mathbf{X}_{(1,2)}(k,\ell) = \calX(i_1,i_2,i_3,i_4), \qquad k = j(i_1,i_2), \quad \ell = j(i_3,i_4),
\end{equation}
where the function $j(\cdot)$ is defined as in~\eqref{eq:defmat}.

\begin{lemma}\label{lemma:squaremat}
Let $\calA \in \R^{n\times n\times n \times n}$ be antisymmetric. Then the following statements hold:
\begin{enumerate}
\item $\mathbf{A}_{(1,2)}$ is symmetric.
\item Let $\lambda$ be a nonzero eigenvalue of $\mathbf{A}_{(1,2)}$ with eigenvector $v \in \R^{n^2}$. Then the matricization $\mathbf{V}_{(1)} \in \R^{n\times n}$
 of $v$ is skew-symmetric.
\item If $\calA = \antisymm(\alpha u_1\otimes u_2\otimes u_3\otimes u_4)$ such that $\alpha\not=0$ and $[u_1,u_2,u_3,u_4]$ is an orthonormal basis then
 $\mathbf{A}_{(1,2)}$ has an eigenvalue $\alpha/12$ of multiplicity $3$, an eigenvalue $-\alpha/12$ of multiplicity $3$, and $n^2-6$ zero eigenvalues. Any eigenvector $v$ belonging to a nonzero eigenvalue satisfies $\text{\rm range}(\mathbf V_{(1)}) = \text{\rm span}\{ u_1,u_2,u_3,u_4\}$.
\end{enumerate}
\end{lemma}

\begin{proof}
1. This statement follows directly from the definition~\eqref{eq:mat12}:
\[
 \mathbf{A}_{(1,2)}(k,\ell) = \calA(i_1,i_2,i_3,i_4) = \calA(i_3,i_4,i_1,i_2) = \mathbf{A}_{(1,2)}(\ell,k).
\]
2. The relation $\mathbf{A}_{(1,2)}v = \lambda v$ implies
\begin{eqnarray*}
\mathbf{V}_{(1)}(i_1,i_2) &=& v(j(i_1,i_2)) = \frac{1}{\lambda} \sum_{i_3,i_4 = 1}^n \mathbf{A}_{(1,2)}(j(i_1,i_2),j(i_3,i_4)) v(j(i_3,i_4)) \\
&=& \frac{1}{\lambda} \sum_{i_3,i_4 = 1}^n \calA(i_1,i_2,i_3,i_4) v(j(i_3,i_4)) \\
&=& -\frac{1}{\lambda} \sum_{i_3,i_4 = 1}^n \calA(i_2,i_1,i_3,i_4) v(j(i_3,i_4)) = -\mathbf{V}_{(1)}(i_2,i_1),
\end{eqnarray*}
which shows that $\mathbf{V}_{(1)}$ is skew-symmetric. \\[0.1cm]
3. By the definition of $\calA$, $\text{range}(\mathbf V_{(1)}) \subset \text{span}\{ u_1,u_2,u_3,u_4\}$ and, together with its skew-symmetry, this implies that $\mathbf V_{(1)}$ is a linear combination of matrices $u_i u_j^T - u_j u_i^T$ for all $i\not = j$. Let $\pi \in S_d$ and set $\sigma = \text{sign}(\pi)$. Using Lemma~\ref{lemma:skewscalar}, we have
\begin{eqnarray*}
&& \mathbf{A}_{(1,2)} \big( u_{\pi(1)} \otimes u_{\pi(2)} - u_{\pi(2)} \otimes u_{\pi(1)} \big) \\
&=& \text{vec}\big( \calA \times_1 u_{\pi(1)} \times_2 u_{\pi(2)} - \calA \times_1 u_{\pi(2)} \times_2 u_{\pi(1)} \big) \\
&=& \frac{\alpha}{24} \big( \sigma u_{\pi(3)} \otimes u_{\pi(4)} - \sigma u_{\pi(4)} \otimes u_{\pi(3)} + \sigma u_{\pi(3)} \otimes u_{\pi(4)} - \sigma u_{\pi(4)} \otimes u_{\pi(3)} \big) \\
&=& \frac{\sigma \alpha}{12}\big( u_{\pi(3)} \otimes u_{\pi(4)} - u_{\pi(4)} \otimes u_{\pi(3)} \big).
\end{eqnarray*}
In turn, there is an eigenspace of dimension three with orthonormal basis
\begin{equation}\label{eq:poseigenspace}
\begin{array}{ll}
 \big( u_1u_2^T - u_2u_1^T + u_3u_4^T - u_4u_3^T \big) / 2, \\
 \big( u_1u_4^T - u_4u_1^T + u_2u_3^T - u_3u_2^T \big) / 2, \\
 \big( u_3u_1^T - u_1u_3^T + u_2u_4^T - u_3u_2^T \big) / 2, \\
\end{array}
\end{equation}
belonging to the eigenvalue $\alpha / 12$. Due to the orthogonality of $u_1,u_2,u_3,u_4$ the range of any linear combination of~\eqref{eq:poseigenspace} equals
$\text{span}\{ u_1,u_2,u_3,u_4\}$.
Analogously, there is an eigenspace of dimension three belonging to the eigenvalue $-\alpha / 12$ with the same property.
\end{proof}

Lemma~\ref{lemma:squaremat}.3 suggests the initialization strategy described in Algorithm~\ref{alg:HOPMinitialization}.

\begin{algorithm}[H] \small
\caption{HOPM initialization strategy for antisymmetric tensor of order $4$}
\label{alg:HOPMinitialization}
\begin{algorithmic}
\STATE Compute eigenvector $v \in\R^{n^2}$ associated with eigenvalue of largest magnitude $\bfA_{(1,2)}$.
\STATE Form $\mathbf{V}_{(1)} \in\R^{n\times n}$ and compute its SVD.
\STATE Return the four leading left singular vectors $u_1,u_2,u_3,u_4$.
\end{algorithmic}
\end{algorithm}

\subsection{Numerical Experiments}

To investigate the difference between the different initializations, we focus our experiments on antisymmetric tensors of order four.

Figure~\ref{fig:HOPMinitialization} shows the approximation errors returned by HOPM initialized with truncated HOSVD or Algorithm~\ref{alg:HOPMinitialization}, using the random antisymmetric tensors described in Section~\ref{sec:approxnumexp}.
HOPM is considered converged when the norm of the gradient of the objective function reaches $10^{-10}$ or below. It can be seen that both initialization strategies appear to work equally well in terms of the final approximation error.

\begin{figure}[h]
    \centering
    \includegraphics[width=0.49\textwidth]{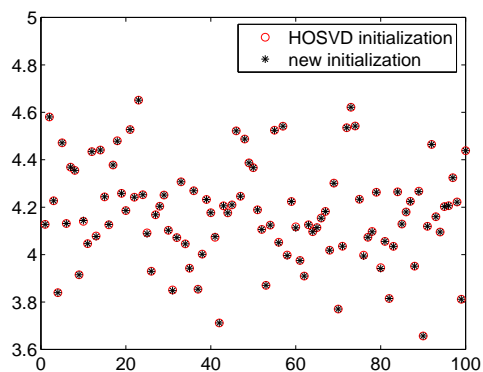}
    \caption{Approximation error of multilinear rank-$4$ approximation produced by HOPM for $100$ random antisymmetric $10\times10\times10\times10$ tensors.}\label{fig:HOPMinitialization}
\end{figure}

Figure~\ref{fig:HOPMgrad} shows the convergence behavior for a typical run.
It turns out that initializing with Algorithm~\ref{alg:HOPMinitialization} gives a significant convergence benefit  both for the approximation error and the norm of the gradient.

\begin{figure}[h]
    \centering
    \begin{subfigure}[b]{0.45\textwidth}
        \includegraphics[width=\textwidth]{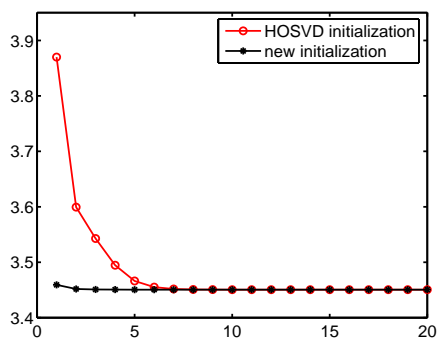}
        \caption{Approximation error}
    \end{subfigure}
    \begin{subfigure}[b]{0.45\textwidth}
        \includegraphics[width=\textwidth]{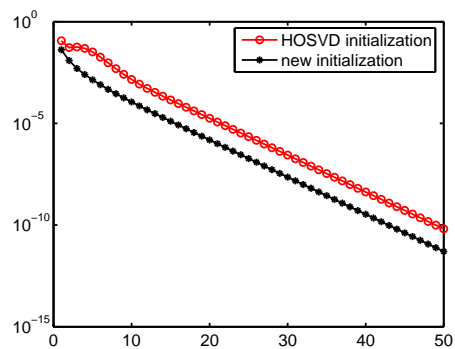}
        \caption{Norm of gradient of objective function}
    \end{subfigure}
    \caption{Convergence behavior of HOPM for multilinear rank-$4$ approximation of random antisymmetric $10\times10\times10\times10$ tensor.}\label{fig:HOPMgrad}
\end{figure}

Finally, analogous to Section~\ref{sec:approxnumexp}, Figure~\ref{fig:Frank1} shows results for the $10\times 10\times 10\times 10$ tensor generated by the function
\[f(x,y,z,w) = \exp( -\sqrt{x^2 + 2 y^2 + 3z^2 + 4w^2}).\]
In this case, both initialization methods yield excellent approximations, with the new initialization resulting in significantly fewer iterations. The same observation can be made in Figure~\ref{fig:beylkin2} for the $9\times 9\times 9\times 9$ tensor corresponding to the antisymmetric ground state described in Section~\ref{sec:approxnumexp}.

\begin{figure}[h]
     \centering
    \begin{subfigure}[b]{0.45\textwidth}
        \includegraphics[width=\textwidth]{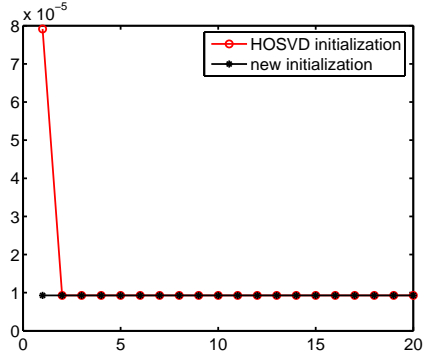}
        \caption{Approximation error}
    \end{subfigure}
    \begin{subfigure}[b]{0.45\textwidth}
        \includegraphics[width=\textwidth]{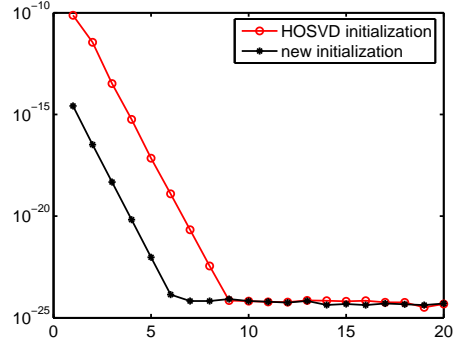}
        \caption{Norm of gradient of objective function}
    \end{subfigure}
    \caption{Convergence behavior of HOPM algorithm for multilinear rank-$4$ approximation of function-related tensor.}\label{fig:Frank1}
\end{figure}

\begin{figure}[h]
    \centering
    \begin{subfigure}[b]{0.45\textwidth}
        \includegraphics[width=\textwidth]{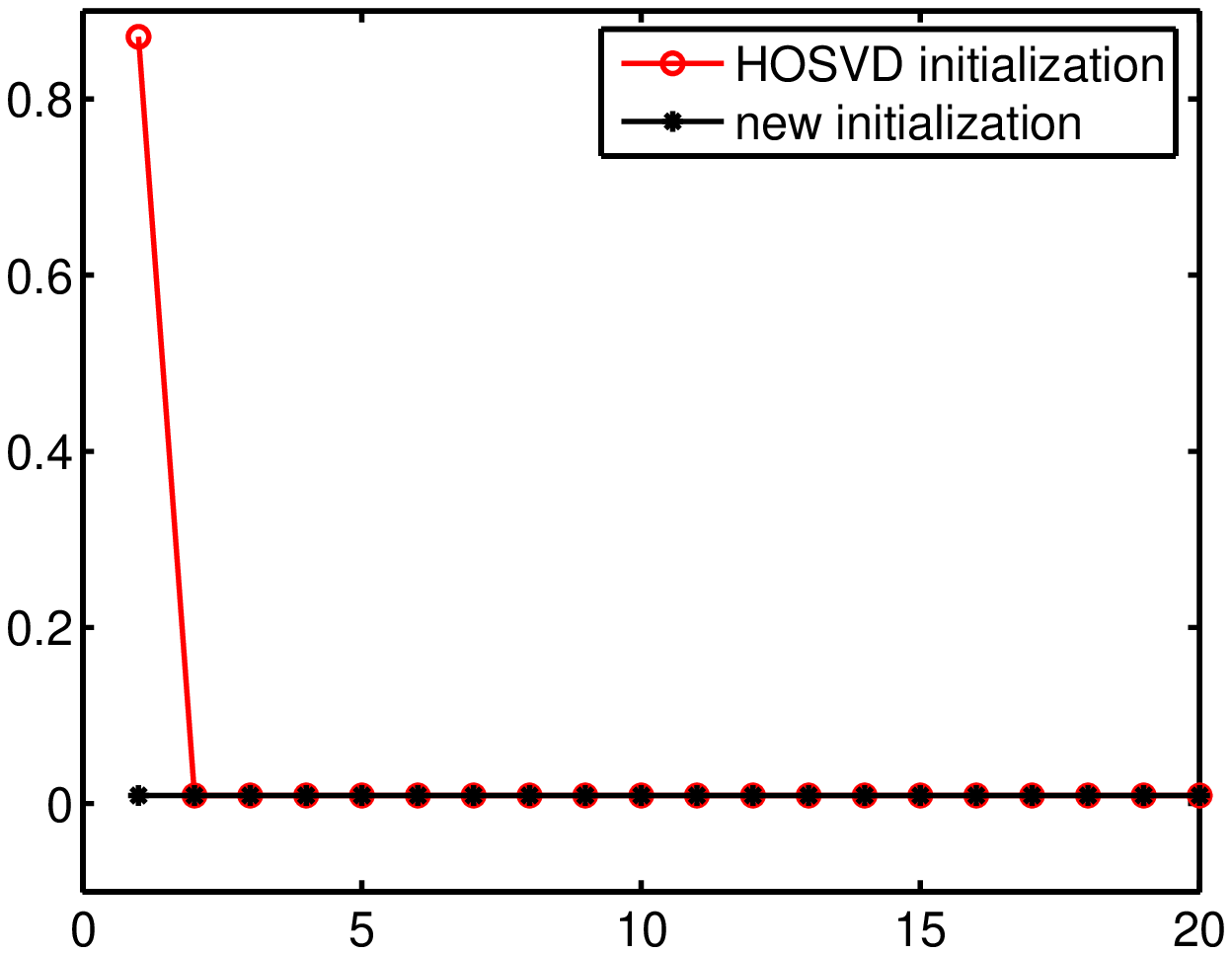}
        \caption{Approximation error}
    \end{subfigure}
    \begin{subfigure}[b]{0.45\textwidth}
        \includegraphics[width=\textwidth]{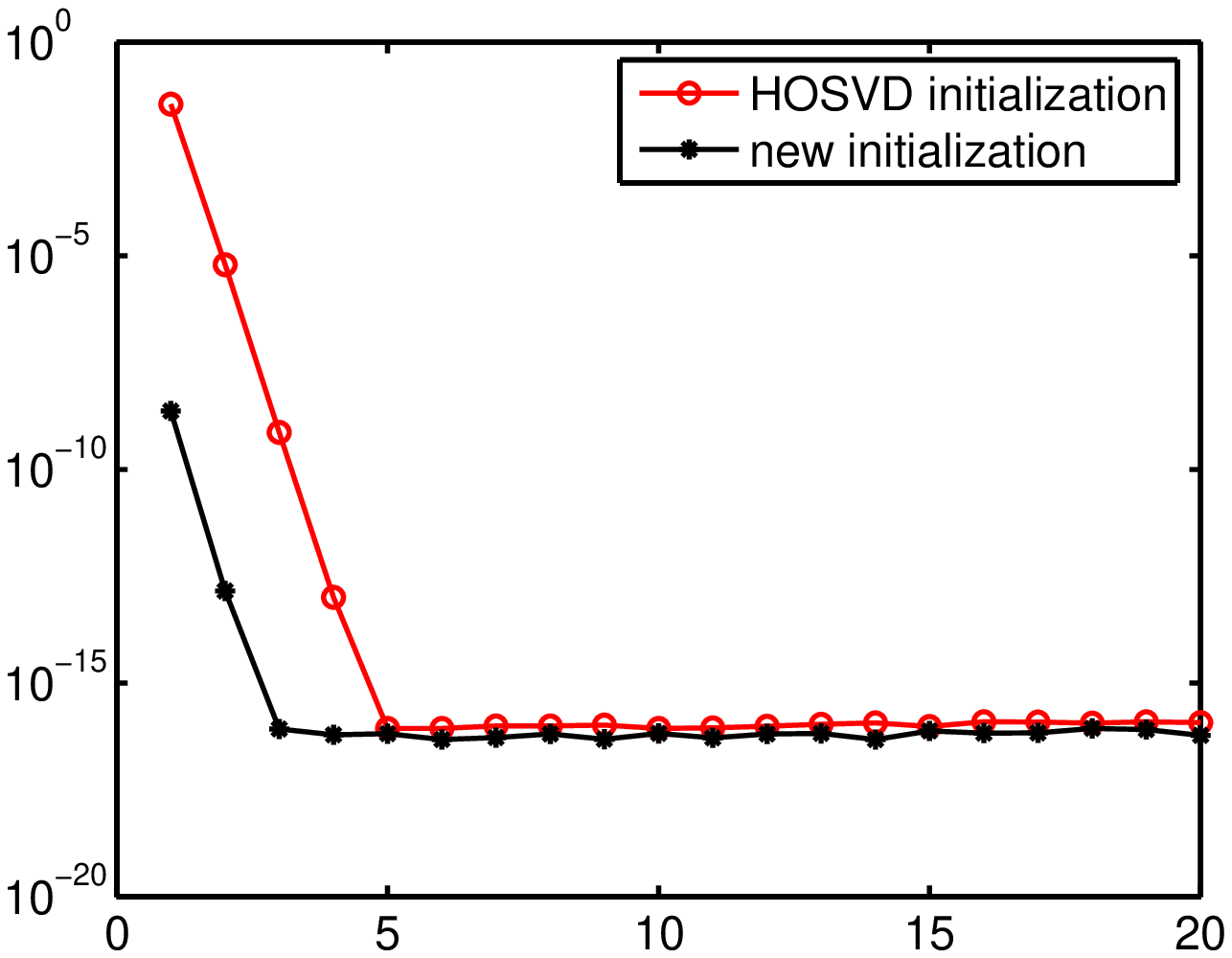}
        \caption{Norm of gradient of objective function}
    \end{subfigure}
    \caption{Convergence behavior of HOPM algorithm for multilinear rank-$4$ approximation of antisymmetric ground state for the discretized Hamiltonian~\eqref{eq:hamiltonian} with $d = 4$ and $n = 9$.}\label{fig:beylkin2}
\end{figure}

\section{Conclusions}

The multilinear rank of an antisymmetric tensor has been analyzed and new algorithms for antisymmetric low multilinear rank approximation have been proposed. The Jacobi algorithm initialized with truncated HOSVD preserves antisymmetry and appears to enjoy excellent global convergence properties. We have shown that a best unstructured rank-$1$ approximation can always be turned into a best antisymmetric multilinear rank-$d$ approximation. In such a scenario,
HOPM initialized either with truncated HOSVD (for $d\not=4$) or Algorithm~\ref{alg:HOPMinitialization} (for $d = 4$) is certainly the method of choice.
The algorithms discussed in this paper could provide a building block in the design of low-rank tensor algorithms~\cite{Grasedyck2013a} for eigenvalue problems
with antisymmetric eigenvectors. In particular, the simplicity of HOPM makes it well suited in the context of truncated iterations and greedy strategies.

\bibliographystyle{plain}

\begin{thebibliography}{00}

\bibitem{Beylkin2005}
\textsc{G. Beylkin, M. J. Mohlenkamp}, \textit{Algorithms for numerical analysis in high dimensions},
SIAM J. Sci. Comput., 26 (6) (2005), pp. 2133--2159.

\bibitem{Beylkin2008}
\textsc{G. Beylkin, M. J. Mohlenkamp, F. P{\'e}rez}, \textit{Approximating a wavefunction as an unconstrained sum of {S}later determinants},
J. Math. Phys., 49 (3) (2008), pp. 032107, 28.

\bibitem{DeLathauwer2000}
\textsc{L. De Lathauwer, B. De Moor, J. Vandewalle}, \textit{A Multilinear Singular Value Decomposition},
SIAM J. Matrix Anal. Appl., 21 (4) (2000), pp. 1253--1278.

\bibitem{LathauwerLowrank2000}
\textsc{L. De Lathauwer, B. De Moor, J. Vandewalle}, \textit{On the best rank-$1$ and rank-$(R_1,\ldots,R_N)$ approximation of higher order tensors},
SIAM J. Matrix Anal. Appl., 21 (4) (2000), pp. 1324--1342.

\bibitem{Friedland2013}
\textsc{S. Friedland}, \textit{Best rank one approximation of real symmetric tensors can be chosen symmetric},
Front. Math. China, 8 (1) (2013), pp. 19--40.

\bibitem{Grasedyck2013a}
\textsc{L. Grasedyck, D. Kressner, C. Tobler}, \textit{A literature survey of low-rank tensor approximation techniques},
GAMM-Mitt., 36 (2013), pp. 53--78.

\bibitem{Hackbusch2016anti}
\textsc{W. Hackbusch}, \textit{On the representation of symmetric and antisymmetric tensors},
Preprint, Max Planck Institute for Mathematics in the Sciences (2016)

\bibitem{IshtevaJac2013}
\textsc{M. Ishteva, P.-A. Absil, P. Van Dooren}, \textit{Jacobi algorithm for the best low multilinear rank approximation of symmetric tensors},
SIAM J. Matrix Anal. Appl., 34 (2) (2013), pp. 651--672.

\bibitem{Kofidis2002}
\textsc{E. Kofidis, P. A. Regalia}, \textit{On the best rank-1 approximation of higher-order supersymmetric tensors},
SIAM J. Matrix Anal. Appl., 23 (3) (2002), pp. 863--884.

\bibitem{Kolda2009}
\textsc{T. G. Kolda, B. W. Bader}, \textit{Tensor Decompositions and Applications},
SIAM Rev., 51 (3) (2009), pp. 455--500.

\bibitem{Oseledets2011a}
\textsc{I. V. Oseledets}, \textit{Tensor-Train Decomposition},
SIAM J. Sci. Comput., 33 (5) (2011), pp. 2295--2317.

\bibitem{SaadHOOI2006}
\textsc{B. N. Sheehan, Y. Saad}, \textit{Higher Order Orthogonal Iteration of Tensors ({HOOI}) and its Relation to {PCA} and {GLRAM}},
Proceedings of the 2007 SIAM International Conference on Data Mining, (2007), pp. 355--365.

\bibitem{Szalay2015}
\textsc{S. Szalay, M. Pfeffer, V. Murg, G. Barcza, F. Verstraete, R. Schneider, \"{O}. Legeza},
\textit{Tensor product methods and entanglement optimization for ab initio quantum chemistry},
International Journal of Quantum Chemistry, 115 (19) (2015), pp. 1342--1391.
\end{thebibliography}

\end{document}